\documentclass[11pt]{amsart}
\usepackage[margin=30mm]{geometry}
\usepackage{amsmath,amssymb}
\usepackage{amsthm}
\usepackage{mathrsfs}

\newtheorem{thm}{Theorem}[section]
\newtheorem{prop}{Proposition}[section]

\newtheorem{defi}{Definition}[section]

\newtheorem{rem}{Remark}[section]

\newtheorem*{ques}{\it{Question}}
\newtheorem*{claim1}{\it{Claim 1}}
\newtheorem*{claim2}{\it{Claim 2}}

\begin{document}

\title{S-limit shadowing implies the denseness of chain components with the shadowing property}
\author{Noriaki Kawaguchi}
\address{Research Institute of Science and Technology, Tokai University, 4-4-1 Kitakaname, Hiratsuka, Kanagawa 259-1292, Japan}
\email{gknoriaki@gmail.com}

\begin{abstract}
For any continuous self-map of a compact metric space, we consider the space of chain components and prove that the s-limit shadowing implies the denseness of chain components with the shadowing property. It gives a partial answer to a question raised by Moothathu [Topology Appl.\:158 (2011) 2232--2239]. We also prove that the s-limit shadowing is not generic in the space of homeomorphisms of a closed differentiable manifold.
\end{abstract}

\maketitle

\markboth{NORIAKI KAWAGUCHI}{S-limit shadowing implies the denseness of chain components with the shadowing property}

\section{Introduction}
{\em Shadowing} as introduced by the works of Anosov and Bowen \cite{A, B2} is a hyperbolic feature of dynamical systems. It generally refers to a property that coarse orbits, or {\em pseudo-orbits}, can be approximated by true orbits and has played an important role in the study of dynamical systems (see \cite{AH, P} for general background). {\em Chain components} as they appear in the so-called fundamental theorem of dynamical systems by Conley \cite{C} are basic objects in the theory of dynamical systems. They are also called as {\em chain transitive components, basic sets, chain recurrent classes, chain classes}, and so on, and there are many studies in the literature focusing on them. The main result of this paper involves with a variation of shadowing, namely {\em s-limit shadowing}, and chain components. It gives a partial answer to a question raised by Moothathu \cite{M}. As a by-product of the proof technique, we also show that the s-limit shadowing is not generic in the space of all homeomorphisms of a closed differentiable manifold.

We begin with several definitions. Throughout, $X$ denotes a compact metric space endowed with a metric $d$.

\begin{defi}
\normalfont
Given a continuous map $f\colon X\to X$ and $\delta>0$, a finite sequence $(x_i)_{i=0}^{k}$ of points in $X$, where $k>0$ is a positive integer, is called a {\em $\delta$-chain} of $f$ if $d(f(x_i),x_{i+1})\le\delta$ for every $0\le i\le k-1$. Let $\xi=(x_i)_{i\ge0}$ be a sequence of points in $X$. For $\delta>0$, $\xi$ is called a {\em $\delta$-pseudo orbit} of $f$ if $d(f(x_i),x_{i+1})\le\delta$ for all $i\ge0$. For $\epsilon>0$, $\xi$ is said to be {\em $\epsilon$-shadowed} by $x\in X$ if $d(f^i(x),x_i)\leq \epsilon$ for all $i\ge 0$. We say that $f$ has the {\em shadowing property} if for any $\epsilon>0$, there is $\delta>0$ such that every $\delta$-pseudo orbit of $f$ is $\epsilon$-shadowed by some point of $X$.
\end{defi}

Let $f\colon X\to X$ be a continuous map. For any $x,y\in X$ and $\delta>0$, the notation $x\rightarrow_\delta y$ means that there is a $\delta$-chain $(x_i)_{i=0}^k$ of $f$ with $x_0=x$ and $x_k=y$. We write $x\rightarrow y$ if $x\rightarrow_\delta y$ for all $\delta>0$. We say that $x\in X$ is a {\em chain recurrent point} for $f$ if $x\rightarrow x$. Let $CR(f)$ denote the set of chain recurrent points for $f$. We define a relation $\leftrightarrow$ in
\[
CR(f)^2=CR(f)\times CR(f)
\]
by: for any $x,y\in CR(f)$, $x\leftrightarrow y$ if and only if $x\rightarrow y$ and $y\rightarrow x$. Note that $\leftrightarrow$ is a closed $(f\times f)$-invariant equivalence relation in $CR(f)^2$ and satisfies $x\leftrightarrow f(x)$ for all $x\in X$. An equivalence class $C$ of $\leftrightarrow$ is called a {\em chain component} for $f$. We regard the quotient space
\[
\mathcal{C}(f)=CR(f)/{\leftrightarrow}
\]
as a space of chain components.

\begin{rem}
\normalfont
A subset $S$ of $X$ is said to be $f$-invariant if $f(S)\subset S$. For a closed $f$-invariant subset $S$ of $X$, we say that $f|_S\colon S\to S$ is {\em chain transitive} if for any $x,y\in S$ and $\delta>0$, there is a $\delta$-chain $(x_i)_{i=0}^k$ of $f|_S$ such that $x_0=x$ and $x_k=y$. The following properties hold:
\begin{itemize}
\item[(1)] $CR(f)=\bigsqcup_{C\in\mathcal{C}(f)}C$,
\item[(2)] Every $C\in\mathcal{C}(f)$ is a closed $f$-invariant subset of $CR(f)$,
\item[(3)] $f|_C\colon C\to C$ is chain transitive for all $C\in\mathcal{C}(f)$.
\end{itemize}
\end{rem}

Given any continuous map $f\colon X\to X$, let
\[
\mathcal{C}_{\rm sp}(f)=\{C\in\mathcal{C}(f)\colon\text{$f|_C\colon C\to C$ has the shadowing property}\}.
\]
In \cite{M}, Moothathu raised the following question (see Section 7 of \cite{M}):
\begin{ques}
For any continuous map $f\colon X\to X$, if $f$ has the shadowing property, then
\[
\mathcal{C}(f)=\overline{\mathcal{C}_{\rm sp}(f)}?
\]
\end{ques}

In order to state the main result of this paper, we recall the definition of s-limit shadowing. As far as the author knows, the name ``s-limit shadowing'' was introduced by Sakai \cite{S}, but it was implicitly considered by Bowen \cite{B1}.  

\begin{defi}
\normalfont
Let $f\colon X\to X$ be a continuous map and let $\xi=(x_i)_{i\ge0}$ be a sequence of points in $X$. For $\delta>0$, $\xi$ is called a {\em $\delta$-limit-pseudo orbit} of $f$ if $d(f(x_i),x_{i+1})\le\delta$ for all $i\ge0$, and
\[
\lim_{i\to\infty}d(f(x_i),x_{i+1})=0.
\]
For $\epsilon>0$, $\xi$ is said to be {\em $\epsilon$-limit shadowed} by $x\in X$ if $d(f^i(x),x_i)\leq \epsilon$ for all $i\ge 0$, and
\[
\lim_{i\to\infty}d(f^i(x),x_i)=0.
\]
We say that $f$ has the {\em s-limit shadowing property} if for any $\epsilon>0$, there is $\delta>0$ such that every $\delta$-limit-pseudo orbit of $f$ is $\epsilon$-limit shadowed by some point of $X$.
\end{defi}

The main result is the following.

\begin{thm}
Given any continuous map $f\colon X\to X$, if $f$ has the s-limit shadowing property, then
\[
\mathcal{C}(f)=\overline{\mathcal{C}_{\rm sp}(f)}.
\]
\end{thm}

\begin{rem}
\normalfont
\begin{itemize}
\item[(1)] For any continuous map $f\colon X\to X$, if $f$ has the s-limit shadowing property, then $f$ has the shadowing property; therefore, Theorem 1.1 gives a partial answer to the question mentioned above. 
\item[(2)] For any subset $S$ of $X$ and $r>0$, we denote by $B_r(S)$ the $r$-neighborhood of $S$:
\[
B_r(S)=\{x\in X\colon d(x,S)\le r\}.
\]
Given any continuous map $f\colon X\to X$ and $C\in\mathcal{C}(f)$, we easily see that for every $\epsilon>0$, there is $\delta>0$ such that
\[
B_\delta(C)\cap D\ne\emptyset\implies D\subset B_\epsilon(C)
\]
for all $D\in\mathcal{C}(f)$. For $C\in\mathcal{C}(f)$, we say that $C$ is {\em isolated} if there exists $r>0$ such that
\[
B_r(C)\cap D=\emptyset
\]
for all $D\in\mathcal{C}(f)$ with $C\ne D$. Let
\[
\mathcal{C}_{\rm iso}(f)=\{C\in\mathcal{C}(f)\colon\text{$C$ is isolated}\}.
\]
If $f$ has the shadowing property, then as shown in \cite{M},
\[
f|_{CR(f)}\colon CR(f)\to CR(f)
\]
also has the shadowing property. It follows that if $f$ has the shadowing property, then
\[
\mathcal{C}_{\rm iso}(f)\subset\mathcal{C}_{\rm sp}(f).
\]
In particular, if $f$ satisfies the shadowing property and $|\mathcal{C}(f)|<\infty$, then
\[
\mathcal{C}(f)=\mathcal{C}_{\rm iso}(f)=\mathcal{C}_{\rm sp}(f).
\]
Even if $f$ has the shadowing property, if $|\mathcal{C}(f)|=\infty$, then $\mathcal{C}(f)\ne\mathcal{C}_{\rm iso}(f)$, and it can happen that
\[
\mathcal{C}(f)\ne\mathcal{C}_{\rm sp}(f).
\]
\item[(3)] A compact metric space $Z$ is called a {\em Cantor space} if $Z$ is perfect, that is, without isolated points, and totally disconnected. Every Cantor space is homeomorphic to the Cantor ternary set. Let $X$ be a Cantor space and let $f=id_X$, the identity map. Then, $f$ has the s-limit shadowing property. We easily see that $X=CR(f)$ and
\[
\leftrightarrow=\{(x,x)\colon x\in X\}.
\]
It follows that
\[
X=\mathcal{C}(f)=\mathcal{C}_{\rm sp}(f),
\]
and so $\mathcal{C}(f)$ is a Cantor space. Note also that $\mathcal{C}_{\rm iso}(f)=\emptyset$.
\item[(4)] In \cite{K}, by using a result of Good and Meddaugh \cite{GM}, the author showed that for any continuous map $f\colon X\to X$, if $f$ satisfies the shadowing property and $\dim{CR(f)}=0$, that is, $CR(f)$ is totally disconnected, then
\[
\mathcal{C}(f)=\overline{\mathcal{C}_{\rm sp}(f)}.
\]
This is another partial answer to the above question. Also in \cite{K}, an example is given of a continuous map $f\colon X\to X$ such that the following properties fold:
\begin{itemize}
\item $X$ is a Cantor space,
\item $f$ has the shadowing property,
\item $\mathcal{C}(f)$ is a Cantor space,
\item $\mathcal{C}(f)=\overline{\mathcal{C}_{\rm sp}(f)}$,
\item $\mathcal{C}_{\rm sp}(f)$ is a countable set and so is a meager subset of $\mathcal{C}(f)$.
\end{itemize}
\item[(5)] Let $M$ be a closed differentiable manifold and let $\mathcal{C}(M)$ (resp.\:$\mathcal{H}(M)$) denote the set of all continuous self-maps (resp.\:homeomorphisms) of $M$, endowed with the $C^0$-topology. Since generic $f\in\mathcal{C}(M)$ (resp.\:$f\in\mathcal{H}(M)$) satisfies the shadowing property and $\dim{CR(f)}=0$ (see Section 1 of \cite{K} for details), by the result in \cite{K} mentioned in the previous remark, such $f$ satisfies
\[
\mathcal{C}(f)=\overline{\mathcal{C}_{\rm sp}(f)}.
\]
Due to results in \cite{AHK}, we know that for generic $f\in\mathcal{C}(M)$ (resp.\:$f\in\mathcal{H}(M)$), $\mathcal{C}(f)$ is a Cantor space and so
\[
\mathcal{C}_{\rm iso}(f)=\emptyset.
\]
We also remark that for generic $f\in\mathcal{C}(M)$ (resp.\:$f\in\mathcal{H}(M)$),
\[
\mathcal{C}(f)=\overline{\mathcal{C}_{\rm sp}(f)}
\]
is a consequence of $\dim{CR(f)}=0$ and a fact proved in \cite{AHK} that
\[
\bigcup\{C\in\mathcal{C}(f)\colon\text{$C$ is terminal}\}
\]
is dense in $CR(f)$. This is because for any $f\in\mathcal{C}(M)$ and $C\in\mathcal{C}(f)$, since $M$ is locally connected, if $C$ is terminal and $\dim{C}=0$, then $C$ is a periodic orbit, or $(C,f|_C)$ is topologically conjugate to an odometer, so we have $C\in\mathcal{C}_{\rm sp}(f)$ in both cases.
\end{itemize}
\end{rem}

Let $M$ be a closed differentiable manifold and let $\mathcal{C}(M)$ (resp.\:$\mathcal{H}(M)$) be the set of all continuous self-maps (resp.\:homeomorphisms) of $M$, endowed with the $C^0$-topology. The other result of this paper is the following.

\begin{thm}
Generic $f\in\mathcal{H}(M)$ does not have the s-limit shadowing property.
\end{thm}

\begin{rem}
\normalfont
In \cite{MO}, the s-limit shadowing is proved to be dense in $\mathcal{C}(M)$. In a recent paper \cite{BCOT}, it is proved that the s-limit shadowing is generic in $\mathcal{C}(S^1)$ where $S^1$ is the unit circle.
\end{rem}

Let us recall the definition of {\em limit shadowing}. Given a continuous map $f\colon X\to X$, a sequence $(x_i)_{i\ge0}$ of points in $X$ is called a {\em limit-pseudo orbit} of $f$ if
\[
\lim_{i\to\infty}d(f(x_i),x_{i+1})=0,
\]
and said to be {\em limit shadowed} by $x\in X$
if
\[
\lim_{i\to\infty}d(f^i(x),x_i)=0.
\]
We say that $f$ has the {\em limit shadowing property} if every limit-pseudo orbit of $f$ is limit shadowed by some point of $X$. We know that if $f$ has the s-limit shadowing property, then $f$ satisfies the limit shadowing property (see \cite{BGO}).

\begin{rem}
\normalfont
In \cite{P}, it is proved that for any $f\in\mathcal{H}(S^1)$, if $f$ has the shadowing property, then $f$ satisfies the limit shadowing property. Since the shadowing is generic in $\mathcal{H}(S^1)$, the limit shadowing is also generic in $\mathcal{H}(S^1)$. Although, by Theorem 1.2, the s-limit shadowing is not generic in $\mathcal{H}(S^1)$.
\end{rem}

This paper consists of three sections. In Section 2, we prove Theorem 1.1. Theorem 1.2 is proved in Section 3.

\section{Proof of Theorem 1.1}

In this section, we prove Theorem 1.1. Given any continuous map $f\colon X\to X$ and $\delta>0$, we define an equivalence relation $\leftrightarrow_\delta$ in
\[
CR(f)^2=CR(f)\times CR(f)
\]
by: for any $x,y\in CR(f)$, $x\leftrightarrow_\delta y$ if and only if $x\rightarrow_\delta y$ and $y\rightarrow_\delta x$. Since $x\leftrightarrow_\delta f(x)$ for every $x\in X$, each equivalence class of $\leftrightarrow_\delta$ is an $f$-invariant subset of $CR(f)$. For any $x,y\in CR(f)$ with $d(x,y)\le\delta$, take $z\in CR(f)$ such that $f(z)=x$.
Then, for every $\delta>0$, there is a $\delta$-chain $(x_i)_{i=0}^k$ of $f$ such that $x_0=x$ and $x_k=z$. Since
\[
(x_0,x_1,\dots,x_k,y)
\]
is a $\delta$-chain of $f$, we obtain $x\rightarrow_\delta y$. Similarly, we obtain $y\rightarrow_\delta x$, thus any $x,y\in CR(f)$ with $d(x,y)\le\delta$ satisfies $x\leftrightarrow_\delta y$. It follows that each equivalence class of $\leftrightarrow_\delta$ is an open and closed subset of $CR(f)$. We denote by $\mathcal{C}_\delta(f)$ the (finite) set of equivalence classes of $\leftrightarrow_\delta$. Since
\[
\leftrightarrow=\bigcap_{\delta>0}\leftrightarrow_\delta,
\]
for any $C\in\mathcal{C}(f)$ and $\delta>0$, there is a unique $C_\delta\in\mathcal{C}_\delta(f)$ such that $C\subset C_\delta$. Then, we have $C_{\delta_1}\subset C_{\delta_2}$ for all $0<\delta_1<\delta_2$, and
\[
C=\bigcap_{\delta>0}C_\delta.
\]
By compactness, for every $\epsilon>0$, we have $C_\delta\subset B_\epsilon(C)$ for some $\delta>0$, here $B_\epsilon(\cdot)$ denotes the $\epsilon$-neighborhood.

Fix $\delta>0$ and $A\in\mathcal{C}_\delta(f)$. Let
\[
\mathcal{S}=\{C\in\mathcal{C}(f)\colon C\subset A\},
\]
and
\[
\mathcal{S}_\epsilon=\{B\in\mathcal{C}_\epsilon(f)\colon B\subset A\}
\]
for every $0<\epsilon\le\delta$. Then, we have
\[
A=\bigsqcup_{C\in\mathcal{S}}C=\bigsqcup_{B\in\mathcal{S}_\epsilon}B
\]
for all $0<\epsilon\le\delta$. We define an order $\le$ on $\mathcal{S}$ by: for any $C_1, C_2\in\mathcal{S}$, $C_1\le C_2$ if and only if for every $\gamma>0$, there is a $\gamma$-chain $(x_i)_{i=0}^k$ of $f$ with $x_0\in C_2$ and $x_k\in C_1$, or equivalently, there are $p\in C_1$ and $q\in C_2$ such that $q\rightarrow p$. For each  $0<\epsilon\le\delta$, we also define an order $\le_\epsilon$ on $\mathcal{S}_\epsilon$ by: for any $B_1, B_2\in\mathcal{S}_\epsilon$, $B_1\le_\epsilon B_2$ if and only if there are $p\in B_1$ and $q\in B_2$ such that $q\rightarrow_\epsilon p$. By Zorn's lemma, we shall prove the following claim.

\begin{claim1}
$(\mathcal{S},\le)$ has a maximal element.
\end{claim1}

\begin{proof}[Proof of Claim 1]
Given any chain $\mathcal{R}$ in $(\mathcal{S},\le)$, let
\[
\mathcal{R}_\epsilon=\{B\in\mathcal{S}_\epsilon\colon\text{there exists $R\in\mathcal{R}$ such that $R\subset B$}\}
\]
for all $0<\epsilon\le\delta$. Then, for each $0<\epsilon\le\delta$, $\mathcal{R}_\epsilon$ is a chain in $(\mathcal{S_\epsilon},\le_\epsilon)$. For all $0<\epsilon\le\delta$, note that $\mathcal{R}_\epsilon$ is a finite set, and let
\[
B_\epsilon=\max{(\mathcal{R}_\epsilon,\le_\epsilon)}.
\]
By compactness, we can take a sequence $\delta\ge\epsilon_1>\epsilon_2>\cdots$ and a closed subset $B$ of $X$ such that
\[
\tag{1} \lim_{n\to\infty}\epsilon_n=0
\]
and
\[
\tag{2} \lim_{n\to\infty}d_H(B_{\epsilon_n}, B)=0,
\]
where $d_H$ is the Hausdorff distance. For all $n\ge1$, since $B_{\epsilon_n}\in\mathcal{R}_{\epsilon_n}\subset\mathcal{S}_{\epsilon_n}$, we have $B_{\epsilon_n}\subset A$. By (2), we obtain $B\subset A$. Since $B_{\epsilon_n}\in\mathcal{R}_{\epsilon_n}\subset\mathcal{S}_{\epsilon_n}\subset\mathcal{C}_{\epsilon_n}(f)$ for any $n\ge1$, by (1) and (2), we obtain $x\leftrightarrow y$ for all $x,y\in B$, thus there is $C\in\mathcal{C}(f)$ with $B\subset C$. It follows that $A\cap C\ne\emptyset$, and this implies $C\subset A$, that is, $C\in\mathcal{S}$. Given any $R\in\mathcal{R}$, take $B^{(n)}\in\mathcal{C}_{\epsilon_n}(f)$, $n\ge1$, such that $R\subset B^{(n)}$ for all $n\ge1$. Then, by (1), we have
\[
\tag{3} R=\bigcap_{n\ge1}B^{(n)}.
\]
Also, we easily see that $B^{(n)}\in\mathcal{R}_{\epsilon_n}$ for all $n\ge1$. For each $n\ge1$, since
\[
B_{\epsilon_n}=\max{(\mathcal{R}_{\epsilon_n},\le_{\epsilon_n})},
\]
there is an $\epsilon_n$-chain $(x_i)_{i=0}^k$ of $f$ with $x_0\in B_{\epsilon_n}$ and $x_k\in B^{(n)}$. 
From $B\subset C$, (1), (2), and (3), it follows that for every $\gamma>0$, there is a $\gamma$-chain $(x_i)_{i=0}^k$ of $f$ with $x_0\in C$ and $x_k\in R$, that is, $R\le C$. Since $R$ is arbitrary, $C$ gives an upper bound of $\mathcal{R}$ in $(\mathcal{S},\le)$. Since $\mathcal{R}$ is arbitrary, by Zorn's lemma, we conclude that $(\mathcal{S},\le)$ has a maximal element, proving the claim.
\end{proof}

Let $C^\ast$ be a maximal element of $(\mathcal{S},\le)$. We shall prove the following claim.

\begin{claim2}
If $f$ has the s-limit shadowing property, then $C^\ast\in\mathcal{C}_{\rm sp}(f)$. 
\end{claim2} 

\begin{proof}[Proof of Claim 2]
As shown above, we have
\[
\tag{4} \delta<\min_{A'\in\mathcal{C}_\delta(f)\setminus\{A\}}\min_{x\in A,y\in A'}d(x,y).
\]
Given any $0<\gamma\le\delta$, since $f$ has the s-limit shadowing property, there is $\beta>0$ such that every $\beta$-limit-pseudo orbit of $f$ is $\gamma$-limit shadowed by some point of $X$. 
Let $\xi=(x_i)_{i\ge0}$ be a $\beta$-pseudo orbit of $f|_{C^\ast}$. For each $k\ge0$, we take $y_k\in C^\ast$ with $f^k(y_k)=x_0$ and consider a $\beta$-limit-pseudo orbit
\[
\xi_k=(y_k,f(y_k),\dots,f^{k-1}(y_k),x_0,x_1,\dots,x_k,f(x_k),f^2(x_k),\dots)
\]
of $f$, which is $\gamma$-limit shadowed by some $z^{(k)}\in X$. Note that for every $k\ge0$,
\[
d(f^i(z^{(k)}),C^\ast)\le\gamma
\]
for all $i\ge0$, and
\[
\tag{5} \lim_{i\to\infty}d(f^i(f^k(z^{(k)})),C^\ast)=0.
\]
By compactness, we have
\[
\lim_{l\to\infty}f^{k_l}(z^{(k_l)})=z
\]
for some sequence $0\le k_1<k_2<\cdots$ and some $z\in X$. By the choice of $\xi_k$, $k\ge0$, we easily see that $\xi$ is $\gamma$-shadowed by $z$. In order to prove $z\in C^\ast$, by taking a subsequence if necessary, we may assume that
\[
\lim_{l\to\infty}f^{j+k_l}(z^{(k_l)})=z_j
\]
for all $j\le0$ for some $z_j\in X$. Then, the sequence $(z_j)_{j\le0}$ satisfies the following properties:
\begin{itemize}
\item $z_0=z$,
\item $f(z_{j-1})=z_j$ for every $j\le0$,
\item $d(z_j,C^\ast)\le\gamma$ for all $j\le0$.
\end{itemize}
By these properties, we obtain $C\in\mathcal{C}(f)$ and $y\in C$ such that
\[
\lim_{j\to-\infty}d(z_j,C)=0,
\]
$d(y,C^\ast)\le\gamma$, and $y\rightarrow z$. By (4), $0<\gamma\le\delta$, and $d(y,C^\ast)\le\gamma$, we obtain $C\subset A$, that is, $C\in\mathcal{S}$. On the other hand, by (5), we have $z\rightarrow w$ for some $w\in C^\ast$. Since $y\rightarrow z$ and $z\rightarrow w$, it holds that $y\rightarrow w$, so by the maximality of $C^\ast$ in $(\mathcal{S},\le)$, $C=C^\ast$. Again by $y\rightarrow z$ and $z\rightarrow w$, we obtain  $z\in C^\ast$. Recall that $\xi$ is $\gamma$-shadowed by $z$. Since $0<\gamma\le\delta$ and $\xi$ are arbitrary, we conclude that $f|_{C^\ast}$ satisfies the shadowing property, that is, $C^\ast\in\mathcal{C}_{\rm sp}(f)$, completing the proof of the claim.
\end{proof}

Suppose that $f$ has the s-limit shadowing property. As shown above, for any $\delta>0$ and $A\in\mathcal{C}_\delta(f)$, there is $C^\ast\in\mathcal{C}(f)$ such that $C^\ast\subset A$ and $C^\ast\in\mathcal{C}_{\rm sp}(f)$. This implies that for any $C\in\mathcal{C}(f)$ and any neighborhood $U$ of $C$, there is $C^\ast\in\mathcal{C}(f)$ such that $C^\ast\subset U$ and $C^\ast\in\mathcal{C}_{\rm sp}(f)$. By the definition of the quotient topology, we conclude that
\[
\mathcal{C}(f)=\overline{\mathcal{C}_{\rm sp}(f)},
\]
thus Theorem 1.1 has been proved.

\section{Proof of Theorem 1.2}

In this section, we prove Theorem 1.2. Given any continuous map $f\colon X\to X$, $C\in\mathcal{C}(f)$ is said to be {\em terminal} if for every $\epsilon>0$, there is $\delta>0$ such that every $\delta$-chain $(x_i)_{i=0}^k$ of $f$ with $x_0\in C$ satisfies $d(x_i,C)\le\epsilon$ for all $0\le i\le k$. Note that if $f$ is a homeomorphism, then
\[
\mathcal{C}(f)=\mathcal{C}(f^{-1}).
\]
We say that $C\in\mathcal{C}(f)$ is {\em initial} if $C$ is terminal for $f^{-1}$. We denote by $\mathcal{C}_{\rm ini}(f)$ the set of initial chain components for $f$.

Let $M$ be a closed differentiable manifold. Due to results in \cite{AHK}, generic $f\in\mathcal{H}(M)$ satisfies the following properties:
\begin{itemize}
\item $CR(f)$ is a  Cantor space and so $M\setminus CR(f)\ne\emptyset$,
\item $\mathcal{C}(f)=\overline{\mathcal{C}_{\rm ini}(f)}$.
\end{itemize}
For such $f\in\mathcal{H}(M)$, to obtain a contradiction, assume that $f$ has the s-limit shadowing property. Take $x\in M\setminus CR(f)$ and $\epsilon>0$ such that $d(x,CR(f))>\epsilon$. By the assumption, we obtain $\delta>0$ such that every $\delta$-limit-pseudo orbit of $f$ is $\epsilon$-limit shadowed by some point of $M$. Let $\omega(x,f)$ denote the $\omega$-limit set of $x$ for $f$. Since
\[
f|_{\omega(x,f)}\colon\omega(x,f)\to\omega(x,f)
\]
is chain transitive, $\omega(x,f)\subset C$ for some $C\in\mathcal{C}(f)$. Then, since $\mathcal{C}(f)=\overline{\mathcal{C}_{\rm ini}(f)}$, there are $C'\in\mathcal{C}_{\rm ini}(f)$, $z\in C$, and $w\in C'$ such that $d(z,w)\le\delta/2$. Since $\omega(x,f)\subset C$ and $z\in C$, we can take a $\delta/2$-chain $(x_i)_{i=0}^k$ of $f$ with $x_0=x$ and $x_k=z$. Consider a $\delta$-limit-pseudo orbit
\[
\xi=(x_0,x_1,\dots,x_{k-1},w,f(w),f^2(w),\dots)
\]
of $f$, which is $\epsilon$-limit shadowed by some $y\in X$. Since $d(x,y)\le\epsilon$, by the the choice of $\epsilon>0$, we obtain $y\notin CR(f)$. On the other hand, by the choice of $\xi$, we have
\[
\lim_{i\to\infty}d(f^i(y),C')=0.
\]
These properties contradicts $C'\in\mathcal{C}_{\rm ini}(f)$, thus we conclude that generic $f\in\mathcal{H}(M)$ does not have the s-limit shadowing property, and Theorem 1.2 has been proved.

By a similar argument as in the proof of Theorem 1.1 in Section 2, we can prove the following proposition. For completeness, we present a proof of it.

\begin{prop}
Given any homeomorphism $f\colon X\to X$, if $f$ has the s-limit shadowing property, then
\[
\mathcal{C}_{\rm ini}(f)\subset\mathcal{C}_{\rm sp}(f).
\]
\end{prop}

\begin{proof}[Proof of Proposition 3.1]
Let $C\in\mathcal{C}_{\rm ini}(f)$. Given any $\epsilon>0$, since $f$ has the s-limit shadowing property, there exists $\delta>0$ such that every $\delta$-limit-pseudo orbit of $f$ is $\epsilon$-limit shadowed by some point of $X$. Let $\xi=(x_i)_{i\ge0}$ be a $\delta$-pseudo orbit of $f|_C$. For each $k\ge0$, we consider a $\delta$-limit-pseudo orbit
\[
\xi_k=(x_0,x_1,\dots,x_k,f(x_k),f^2(x_k),\dots)
\]
of $f$, which is $\epsilon$-limit shadowed by some $x^{(k)}\in X$. By the choice of $\xi_k$, $k\ge0$, we obtain
\[
\tag{1} \lim_{i\to\infty}d(f^i(x^{(k)}),C)=0
\]
for all $k\ge0$. By compactness, we have
\[
\lim_{l\to\infty}x^{(k_l)}=x
\]
for some sequence $0\le k_1<k_2<\cdots$ and some $x\in X$. Again by the choice of $\xi_k$, $k\ge0$, we easily see that $\xi$ is $\epsilon$-shadowed by $x$. On the other hand, by $C\in\mathcal{C}_{\rm ini}(f)$ and (1), we have $x^{(k)}\in C$ for all $k\ge0$, which implies $x\in C$. Since $\epsilon>0$ and $\xi$ are arbitrary, we conclude that $f|_C$ have the shadowing property, that is, $C\in\mathcal{C}_{\rm sp}(f)$. Since $C\in\mathcal{C}_{\rm ini}(f)$ is arbitrary, we obtain
\[
\mathcal{C}_{\rm ini}(f)\subset\mathcal{C}_{\rm sp}(f),
\]
completing the proof
\end{proof}


\begin{thebibliography}{99}

\bibitem{AHK} E. Akin, M. Hurley, J. Kennedy, Dynamics of topologically generic homeomorphisms. Mem. Amer. Math. Soc. 164 (2003).

\bibitem{A} D.V.\:Anosov, Geodesic flows on closed Riemann manifolds with negative curvature. Proc. Steklov Inst. Math. 90 (1967), 235 p.

\bibitem{AH} N.\:Aoki, K.\:Hiraide, Topological theory of dynamical systems. Recent advances. North--Holland Mathematical Library, 52. North--Holland Publishing Co., 1994.

\bibitem{BGO} A.D.\:Barwell, C.\:Good, P.\:Oprocha, Shadowing and expansivity in subspaces. Fund. Math. 219 (2012), 223--243.

\bibitem{BCOT} J.\:Bobok, J.\:\v{C}in\v{c}, P.\:Oprocha, S.\:Troubetzkoy, S-limit shadowing is generic for continuous
Lebesgue measure-preserving circle maps, Ergodic Theory Dynam. Systems 43 (2023), 78--98.

\bibitem{B1} R.\:Bowen, $\omega$-limit sets for axiom A diffeomorphisms. J. Differ. Equations 18 (1975), 333--339.

\bibitem{B2} R.\:Bowen, Equilibrium states and the ergodic theory of Anosov diffeomorphisms. Lecture Notes in Mathematics, 470. Springer--Verlag, 1975.

\bibitem{C} C.\:Conley, Isolated invariant sets and the Morse index. CBMS Regional Conference Series
in Mathematics, 38. American Mathematical Society, Providence, R.I., 1978.

\bibitem{GM} C.\:Good, J.\:Meddaugh, Shifts of finite type as fundamental objects in the theory of shadowing. Invent. Math. 220 (2020), 715--736.

\bibitem{K} N.\:Kawaguchi, On $C^0$-genericity of distributional chaos, Ergodic Theory Dynam. Systems 43 (2023), 615--645.

\bibitem{MO} M. Mazur, P. Oprocha, S-limit shadowing is $C^0$-dense. J. Math. Anal. Appl. 408 (2013), 465--475.

\bibitem{M} T.K.S.\:Moothathu, Implications of pseudo-orbit tracing property for continuous maps on compacta. Topology Appl. 158 (2011), 2232--2239.

\bibitem{P} S.Yu.\:Pilyugin, Shadowing in dynamical systems. Lecture Notes in Mathematics, 1706. Springer--Verlag, 1999.

\bibitem{S} K.\:Sakai, Various shadowing properties for positively expansive maps. Topology Appl. 131 (2003), 15--31.

\end{thebibliography}
\end{document}